\documentclass[12pt]{amsart}

\usepackage[T1]{fontenc}
\usepackage[utf8]{inputenc}
\usepackage{lmodern}
\usepackage{cite}
\usepackage{amsmath,amssymb,amsthm,amscd,mathtools,mathrsfs}
\numberwithin{equation}{section}

\newcommand{\PP}{\mathbb{P}}
\newcommand{\C}{\mathbb{C}}
\newcommand{\Q}{\mathbb{Q}}
\newcommand{\Z}{\mathbb{Z}}
\newcommand{\OO}{\mathcal{O}}
\newcommand{\I}{\mathcal{I}}
\newcommand{\M}{\mathcal{M}}
\newcommand{\Hilb}{\operatorname{Hilb}}
\newcommand{\ch}{\operatorname{ch}}
\newcommand{\Ext}{\operatorname{Ext}}
\newcommand{\Hom}{\operatorname{Hom}}
\newcommand{\Pic}{\operatorname{Pic}}

\newcommand{\td}{\operatorname{td}}
\newcommand{\rk}{\operatorname{rk}}
\newcommand{\vdim}{\operatorname{vdim}}
\newcommand{\Ttw}{\mathsf{T}}
\newcommand{\vfc}[1]{[#1]^{\mathrm{vir}}}
\newcommand{\dual}{\vee}

\theoremstyle{plain}
\newtheorem{theorem}{Theorem}[section]
\newtheorem{proposition}[theorem]{Proposition}
\newtheorem{lemma}[theorem]{Lemma}
\newtheorem{corollary}[theorem]{Corollary}
\theoremstyle{definition}
\newtheorem{definition}[theorem]{Definition}
\theoremstyle{remark}
\newtheorem{remark}[theorem]{Remark}

\title[Planar rank-one sheaves on $\PP^3$]{Planar rank-one sheaves on $\PP^3$, obstruction bundles, and divisor-supported Donaldson--Thomas series}
\author[Reginald Anderson]{Reginald Anderson}
\address{Department of Mathematics, University of California, Irvine, 340 Rowland Hall, Irvine, CA 92697-3875, USA}
\email{reginala@uci.edu}
\subjclass[2020]{14N35, 14C05, 14J60, 14F06}
\keywords{Donaldson--Thomas theory, Hilbert schemes of points, obstruction bundles, planar sheaves, Fano threefolds}
\date{}

\begin{document}

\begin{abstract}
Let $X=\PP^3$ and let
\[
\alpha_n=(0,1,-\tfrac12,\tfrac16-n)\in H^{\mathrm{even}}(X,\Q)
\]
with respect to the basis $1,H,H^2,H^3$, where $H=c_1(\OO_X(1))$.
We prove that every Gieseker semistable sheaf on $X$ with Chern character $\alpha_n$ is stable and uniquely of the form $\iota_{P*}\I_Z$ for a plane $P\subset X$ and a length-$n$ subscheme $Z\subset P$.
Hence the coarse stable-sheaf moduli space is the relative Hilbert scheme of points on the universal plane over the dual projective space $X^\dual$.
Using the standard perfect obstruction theory for stable sheaves on a Fano threefold, we identify the obstruction bundle with a relative Carlsson--Okounkov twisted tangent bundle and obtain a closed product formula
\[
\sum_{n\ge 0} \Gamma_n q^n=\prod_{m\ge 1}(1-q^m)^{-7}
\]
for the natural point-inserted two-dimensional Donaldson--Thomas invariants on $\PP^3$.
We then develop the analogous divisor-supported theory for a smooth divisor $D$ in a smooth projective Fano threefold, distinguishing moving and rigid divisors.
For a rigid divisor satisfying $H^1(D,\OO_D)=0$, the divisor-supported moduli component is a Hilbert scheme of points on $D$, its obstruction bundle is the twisted tangent bundle $\Ttw_D^{[n]}(N_{D/X})$, and its generating series is
\[
\prod_{m\ge 1}(1-q^m)^{-(c_2(X)\cdot D+D^3)}.
\]
We work out the examples of the exceptional divisor in $\operatorname{Bl}_p\PP^3$ and of a rigid section in a Fano $\PP^1$-bundle over $\PP^1\times\PP^1$.
\end{abstract}

\maketitle

\section{Introduction}

Donaldson--Thomas theory was originally formulated as a virtual count of stable sheaves on a smooth projective threefold, using the trace-free deformation-obstruction theory of Thomas \cite{Thomas} and the virtual fundamental class of Behrend--Fantechi \cite{BF}.
For zero-dimensional sheaves this gives the degree-zero Donaldson--Thomas theory of Hilbert schemes of points on threefolds, whose product formula is governed by the MacMahon function in the work of Maulik--Nekrasov--Okounkov--Pandharipande and Li \cite{MNOP,Li}.
For two-dimensional sheaves the geometry is more sensitive to the support divisor.
Even when the support is smooth and the sheaf is rank one on it, the threefold obstruction theory is not simply the obstruction theory of the Hilbert scheme of points on the support surface: deformations normal to the surface and their obstructions must be tracked inside the ambient threefold.

A recurring theme in surface-supported Donaldson--Thomas theory is that, after isolating a fixed divisor or a linear system of divisors, the remaining calculation can often be expressed in terms of Hilbert schemes of points on surfaces and tautological obstruction bundles.
This viewpoint appears, for example, in the work of Gholampour--Sheshmani on two-dimensional sheaves and modular forms, and in their later treatment of linear systems and punctual Hilbert schemes \cite{GSmod,GSlinear}.
On the surface side, the relevant universal integral is the Carlsson--Okounkov formula for Euler classes of twisted tangent bundles on Hilbert schemes of points \cite{CO}, building on the classical smoothness and structure of Hilbert schemes of points on surfaces due to Fogarty and G\"ottsche \cite{Fogarty,Gottsche}.

The purpose of the present paper is to give a completely explicit stable-sheaf calculation in the simplest moving-divisor situation: rank-one sheaves supported on planes in $\PP^3$.
The geometric reduction to Hilbert schemes of points on planes is natural, but the point of the paper is to make precise how this reduction interacts with the threefold Donaldson--Thomas obstruction theory as the plane varies.
The result is a closed eta-product series obtained directly from the standard stable-sheaf virtual class, with the moving support family accounted for by a point insertion on the dual projective space.
The same deformation-theoretic mechanism then gives a uniform formula for rigid smooth divisors $D$ in Fano threefolds satisfying $H^1(D,\OO_D)=0$.

Let $X=\PP^3$ and let $H=c_1(\OO_X(1))$.
We study Gieseker semistable sheaves on $X$ with Chern character
\begin{equation}\label{eq:intro-alpha}
\alpha_n=(0,1,-\tfrac12,\tfrac16-n),\qquad n\ge 0.
\end{equation}
These are pure two-dimensional sheaves with support class $H$.
Our first result shows that there is no hidden semistable or non-planar contribution: every such sheaf is stable and is uniquely of the form
\[
\iota_{P*}\I_Z,
\]
where $P\subset \PP^3$ is a plane and $Z\subset P$ is a length-$n$ subscheme.
Equivalently, if
\[
B:=X^\dual=(\PP^3)^\dual
\]
is the dual projective space parametrizing planes and
\[
q\colon \mathcal U\longrightarrow B
\]
is the universal plane, then the coarse stable-sheaf moduli space is the relative Hilbert scheme
\[
M_n:=\Hilb^n(\mathcal U/B).
\]

The scheme $M_n$ is smooth of dimension $2n+3$, while the stable-sheaf obstruction theory on $X$ has virtual dimension $3$.
The dimension difference is exactly supplied by a rank-$2n$ obstruction bundle coming from the normal direction to the universal plane.
Writing $\rho\colon M_n\to B$ for the support-plane map, set
\[
\mathcal U_n:=M_n\times_B\mathcal U,
\qquad
r\colon\mathcal U_n\to\mathcal U,
\qquad
\pi\colon\mathcal U_n\to M_n,
\]
and let $\mathcal I$ be the universal ideal sheaf on $\mathcal U_n$.
We identify the obstruction bundle with a relative Carlsson--Okounkov twisted tangent bundle:
\[
\mathrm{Ob}_n\cong
\mathcal Ext^1_{\pi}\bigl(\mathcal I,\mathcal I\otimes r^*N_{\mathcal U/(B\times X)}\bigr)
\cong
\Ttw_{\mathcal U/B}^{[n]}\bigl(\OO_{\mathcal U/B}(1)\bigr)\otimes \rho^*\OO_B(1).
\]
Consequently
\[
\vfc{M_n}=c_{2n}(\mathrm{Ob}_n)\cap [M_n].
\]
This is the main deformation-theoretic input: it compares the degree-$[1,2]$ truncated Ext theory of rank-zero sheaves on $\PP^3$ with the twisted tangent class on the Hilbert scheme of points on the moving support surface.

\subsection{Comparison with point Hilbert schemes on a threefold}
The series obtained here should not be confused with either of the standard point-counting series on a smooth threefold $Y$.
The ordinary topological Euler characteristic series is
\begin{equation}\label{eq:intro-cheah}
\sum_{n\ge 0} e\bigl(\Hilb^n(Y)\bigr)q^n=M(q)^{e(Y)},
\end{equation}
where
\[
M(q)=\prod_{m\ge 1}(1-q^m)^{-m}
\]
is the MacMahon function; see Cheah's formula for the Hodge polynomial of Hilbert schemes of points \cite{Cheah}.
By contrast, the degree-zero Donaldson--Thomas series is
\begin{equation}\label{eq:intro-dt0}
\sum_{n\ge 0}\deg\vfc{\Hilb^n(Y)}\,q^n=M(-q)^{\int_Y(c_3(Y)-c_1(Y)c_2(Y))},
\end{equation}
conjectured by Maulik--Nekrasov--Okounkov--Pandharipande and proved by Li \cite{Li,MNOP}; see also \cite[Remark~3.8]{BBS}.
For $Y=\PP^3$ this gives
\[
\sum_{n\ge 0} e\bigl(\Hilb^n(\PP^3)\bigr)q^n=M(q)^4,
\qquad
\sum_{n\ge 0}\deg\vfc{\Hilb^n(\PP^3)}\,q^n=M(-q)^{-20}
\]
with
\[
M(-q)^{-20}=1+20q+150q^2+400q^3-855q^4+\cdots.
\]
The planar theory treated here is a two-dimensional sheaf theory rather than a degree-zero theory.
Its natural series is not a MacMahon power, but an eta-power coming from Hilbert schemes of points on $\PP^2$.

\subsection{Main results}
We write $\M_X^{\mathrm{ss}}(\alpha_n)$ for the Simpson coarse moduli scheme of Gieseker semistable sheaves on $X$ with Chern character $\alpha_n$ and polarization $H$.
Our first theorem is the classification statement.

\begin{theorem}[Classification]\label{thm:A}
For every $n\ge 0$, every sheaf in $\M_X^{\mathrm{ss}}(\alpha_n)$ is Gieseker stable and uniquely of the form $\iota_{P*}\I_Z$, where $P\subset \PP^3$ is a plane and $Z\subset P$ is a length-$n$ subscheme.
Consequently
\[
\M_X^{\mathrm{ss}}(\alpha_n)=\M_X^{\mathrm{st}}(\alpha_n)\cong M_n=\Hilb^n(\mathcal U/B),
\]
and $M_n$ is smooth of dimension $2n+3$.
\end{theorem}

Let $\rho\colon M_n\to B$ be the support-plane map.
The universal plane is a smooth divisor in $B\times X$ with normal bundle
\begin{equation}\label{eq:intro-normal}
N_{\mathcal U/(B\times X)}\cong \OO_{\mathcal U/B}(1)\otimes q^*\OO_B(1).
\end{equation}
If $\pi\colon M_n\times_B\mathcal U\to M_n$ and $r\colon M_n\times_B\mathcal U\to\mathcal U$ denote the two projections and $\mathcal I$ is the universal ideal sheaf, then the obstruction bundle is obtained from the groups $\Ext^1_P(\I_Z,\I_Z(1))$.

\begin{theorem}[Obstruction bundle]\label{thm:B}
Under the standard truncated perfect obstruction theory for rank-zero stable sheaves on $X$, the obstruction sheaf on $M_n$ is the rank-$2n$ vector bundle
\[
\mathrm{Ob}_n\cong \mathcal Ext^1_{\pi}\bigl(\mathcal I,\mathcal I\otimes r^*N_{\mathcal U/(B\times X)}\bigr)
\cong
\Ttw_{\mathcal U/B}^{[n]}\bigl(\OO_{\mathcal U/B}(1)\bigr)\otimes \rho^*\OO_B(1).
\]
Moreover
\[
\vfc{M_n}=c_{2n}(\mathrm{Ob}_n)\cap [M_n],
\qquad
\vdim M_n=3.
\]
\end{theorem}

The natural numerical invariant is obtained by integrating the point class from the base $B$.
If $h=c_1(\OO_B(1))$, define
\begin{equation}\label{eq:intro-gamma}
\Gamma_n:=\int_{\vfc{M_n}}\rho^*(h^3).
\end{equation}
The insertion $h^3$ is precisely what removes the three-dimensional freedom of moving the support plane.

\begin{theorem}[Planar product formula]\label{thm:C}
The planar Donaldson--Thomas invariants satisfy
\[
\Gamma_n=\int_{(\PP^2)^{[n]}} c_{2n}\Bigl(\Ttw_{\PP^2}^{[n]}\bigl(\OO_{\PP^2}(1)\bigr)\Bigr),
\]
and their generating series has the closed form
\begin{equation}\label{eq:intro-planar-series}
\sum_{n\ge 0}\Gamma_nq^n=\prod_{m\ge 1}(1-q^m)^{-7}.
\end{equation}
Equivalently,
\[
\sum_{n\ge 0}\rho_*\vfc{M_n}\,q^n=
\left(\prod_{m\ge 1}(1-q^m)^{-7}\right)[B]
\in A_3(B)[[q]].
\]
Thus the planar series is, up to the usual prefactor, the modular form $\eta(\tau)^{-7}$.
\end{theorem}

The paper ends with the corresponding divisor-supported theory on a smooth divisor $D\subset X$ in a Fano threefold.
For rigid divisors satisfying $H^1(D,\OO_D)=0$, the moduli component is literally a Hilbert scheme of points on $D$.

\begin{theorem}[Rigid divisor series]\label{thm:D}
Let $X$ be a smooth projective Fano threefold, let $\iota\colon D\hookrightarrow X$ be a smooth divisor, let $N=N_{D/X}$, and fix a line bundle $L\in\Pic(D)$.
Assume that $D$ is rigid, equivalently $H^0(D,N)=0$, and that $H^1(D,\OO_D)=0$.
For each $n\ge 0$, let $\M_{D,L,n}(X)$ be the stable-sheaf moduli component whose closed points are the sheaves $\iota_*(\I_Z\otimes L)$ with $Z\in D^{[n]}$.
Then
\[
\M_{D,L,n}(X)\cong D^{[n]},
\qquad
\vfc{\M_{D,L,n}(X)}=c_{2n}\Bigl(\Ttw_D^{[n]}(N)\Bigr)\cap [D^{[n]}],
\]
where $\Ttw_D^{[n]}(N)$ is a rank-$2n$ vector bundle.
Its generating series is
\begin{equation}\label{eq:intro-rigid-series}
\sum_{n\ge 0}q^n\int_{\vfc{\M_{D,L,n}(X)}}1
=
\sum_{n\ge 0}q^n\int_{D^{[n]}}e\Bigl(\Ttw_D^{[n]}(N)\Bigr)
=
\prod_{m\ge 1}(1-q^m)^{-(c_2(X)\cdot D+D^3)}.
\end{equation}
\end{theorem}

The planar case is the moving divisor $D=H\subset \PP^3$, for which $h^0(D,N)=3$ and the point insertion in \eqref{eq:intro-gamma} accounts for the three-dimensional support family.
The rigid examples of \S\ref{sec:examples} give the exponents $1$ and $2$ in \eqref{eq:intro-rigid-series}.

\section{Planar sheaves on $\PP^3$}\label{sec:planar}

Let $X=\PP^3$, let $H=c_1(\OO_X(1))$, and let
\[
\alpha_n=(0,1,-\tfrac12,\tfrac16-n).
\]
We work throughout with Gieseker stability with respect to $H$.

\subsection{Pushforward from a plane}
Let $\iota\colon P\hookrightarrow X$ be the inclusion of a plane $P\cong \PP^2$ and let $Z\subset P$ be a length-$n$ subscheme.
Since
\[
\ch(\I_Z)=1-n[\mathrm{pt}]\in H^{\mathrm{even}}(P,\Q)
\]
and $N_{P/X}\cong \OO_P(1)$, Grothendieck--Riemann--Roch yields
\begin{equation}\label{eq:grr-plane}
\ch(\iota_*\I_Z)=\iota_*\bigl((1-n[\mathrm{pt}])\,\td(\OO_P(1))^{-1}\bigr)
=(0,H,-\tfrac12H^2,(\tfrac16-n)H^3).
\end{equation}
Thus the class \eqref{eq:intro-alpha} is exactly the class of a rank-one torsion-free sheaf on a plane, pushed forward to $X$.

\subsection{A stability lemma on a surface}
The following standard lemma will be used repeatedly.

\begin{lemma}\label{lem:rank-one-stable}
Let $S$ be a smooth projective surface, let $A$ be an ample divisor on $S$, and let $E$ be a rank-one torsion-free sheaf on $S$.
Then $E$ is Gieseker stable with respect to $A$.
In particular, for every zero-dimensional subscheme $Z\subset S$, the ideal sheaf $\I_Z$ is Gieseker stable.
\end{lemma}

\begin{proof}
Let $F\subsetneq E$ be a proper nonzero subsheaf.
Since $E$ is torsion-free of rank one, so is $F$.
Write $Q:=E/F$, so that $Q\neq 0$ has dimension at most one.
If $\dim Q=1$, then $c_1(F)=c_1(E)-C$ for a nonzero effective divisor class $C$.
The reduced Hilbert polynomial of a rank-one torsion-free sheaf on a surface has the form
\[
\frac{A^2}{2}m^2+\Bigl(c_1(\cdot)\cdot A-\frac{K_S\cdot A}{2}\Bigr)m+\chi(\cdot),
\]
so the coefficient of $m$ is strictly smaller for $F$ than for $E$.
Hence $p_F(m)<p_E(m)$ for $m\gg 0$.
If $\dim Q=0$, then $c_1(F)=c_1(E)$ and $\chi(F)=\chi(E)-\operatorname{length}(Q)<\chi(E)$, so again $p_F(m)<p_E(m)$ for $m\gg 0$.
Therefore $E$ is Gieseker stable.
\end{proof}

\subsection{Classification and moduli}
We now identify the moduli of sheaves of class $\alpha_n$.

For a locally Noetherian $\C$-scheme $T$, let $\mathscr H_n(T)$ be the set of pairs
\[
(a,\mathcal Z),\qquad
a\colon T\to B,
\qquad
\mathcal Z\subset \mathcal U_T:=T\times_{B,a}\mathcal U,
\]
such that $\mathcal Z\to T$ is finite and flat of degree $n$.
This is the relative Hilbert functor represented by $\Hilb^n(\mathcal U/B)$.
Let $\mathscr M_{\alpha_n}(T)$ be the set of equivalence classes of coherent sheaves $\mathcal F$ on $X\times T$ that are flat over $T$ and whose geometric fibers are Gieseker stable of Chern character $\alpha_n$, with two families equivalent if
\[
\mathcal F'\cong \mathcal F\otimes p_T^*A
\qquad\text{for }A\in\Pic(T),
\]
where $p_T\colon X\times T\to T$ is the projection.
Thus $\mathscr M_{\alpha_n}$ is the scalar-rigidified stable-sheaf functor; the unrigidified stable-sheaf stack retains the scalar automorphism group $\mathbb G_m$.
The Simpson moduli scheme $\M_X^{\mathrm{st}}(\alpha_n)$ corepresents this functor; see \cite[Theorem~1.21]{Simpson} and \cite[Section~3]{Thomas}.

\begin{theorem}\label{thm:classification}
Every Gieseker semistable sheaf $F$ on $X=\PP^3$ with $\ch(F)=\alpha_n$ is Gieseker stable and uniquely of the form
\[
F\cong \iota_{P*}\I_Z,
\]
where $P\subset X$ is a plane and $Z\subset P$ is a length-$n$ subscheme.
Consequently
\[
\M_X^{\mathrm{ss}}(\alpha_n)=\M_X^{\mathrm{st}}(\alpha_n)\cong \Hilb^n(\mathcal U/B).
\]
\end{theorem}

\begin{proof}
Let $F$ be Gieseker semistable with $\ch(F)=\alpha_n$.
Since $\rk(F)=0$ and $c_1(F)=H$, the sheaf $F$ has dimension two.
Semistability implies purity, so $F$ has no embedded components of smaller dimension.
For a pure two-dimensional sheaf on a smooth threefold, the fundamental cycle equals $c_1(F)$.
Thus the support cycle of $F$ is $H$, the class of a plane with multiplicity one.
Therefore the scheme-theoretic support of $F$ is a unique plane $P\subset X$.
Writing $\iota\colon P\hookrightarrow X$ for the inclusion, there is a unique coherent sheaf $E$ on $P$ with $F\cong \iota_*E$.
The generic multiplicity of $F$ along $P$ is one, so $E$ has rank one on $P$.
Applying \eqref{eq:grr-plane} in reverse shows that $\ch(E)=(1,0,-n)$ on $P$.
Since $P\cong \PP^2$ is smooth, $E^{\dual\dual}$ is a line bundle of first Chern class $0$, hence $E^{\dual\dual}\cong \OO_P$.
The quotient $E^{\dual\dual}/E$ is zero-dimensional of length $n$, so $E\cong \I_Z$ for a unique $Z\in P^{[n]}$.
This proves the classification.

By Lemma~\ref{lem:rank-one-stable}, the sheaf $\I_Z$ is Gieseker stable on $P$.
Because
\[
\chi\bigl(X,\iota_*G(m)\bigr)=\chi\bigl(P,G(m)\bigr)
\]
for every coherent sheaf $G$ on $P$, the reduced Hilbert polynomial of $\iota_*G$ on $X$ agrees with that of $G$ on $P$.
Therefore $\iota_*\I_Z$ is Gieseker stable on $X$.
Hence semistability equals stability in the class $\alpha_n$.

It remains to prove the asserted scheme-theoretic identification.
Set
\[
H_n:=\Hilb^n(\mathcal U/B),
\qquad
\mathcal U_{H_n}:=H_n\times_B\mathcal U,
\]
and let $\mathcal Z\subset\mathcal U_{H_n}$ be the universal length-$n$ subscheme with ideal sheaf $\mathcal I_{\mathcal Z}$.
If
\[
j\colon \mathcal U_{H_n}\hookrightarrow H_n\times X
\]
is the natural closed immersion, define
\[
\mathcal F_{H_n}:=j_*\mathcal I_{\mathcal Z}.
\]
The morphism $\mathcal U_{H_n}\to H_n$ is a smooth $\PP^2$-bundle, so $\mathcal O_{\mathcal U_{H_n}}$ is flat over $H_n$, while $\mathcal O_{\mathcal Z}$ is flat over $H_n$ by the defining property of the relative Hilbert scheme.
The exact sequence
\[
0\longrightarrow\mathcal I_{\mathcal Z}
\longrightarrow\mathcal O_{\mathcal U_{H_n}}
\longrightarrow\mathcal O_{\mathcal Z}
\longrightarrow 0
\]
therefore shows that $\mathcal I_{\mathcal Z}$ is flat over $H_n$.
Since $j$ is a closed immersion over $H_n$, the sheaf $\mathcal F_{H_n}$ is also flat over $H_n$, and its formation commutes with arbitrary base change.
Its geometric fiber at $(P,Z)$ is $\iota_{P*}\mathcal I_Z$, which is stable of class $\alpha_n$ by the preceding argument.
More generally, the construction
\[
(a,\mathcal Z)\longmapsto (j_T)_*\mathcal I_{\mathcal Z}
\]
is compatible with base change and defines a natural transformation $\mathscr H_n\to\mathscr M_{\alpha_n}$.
Applied to the universal relative-Hilbert family, it gives a classifying morphism
\[
\Phi\colon H_n\longrightarrow \M_X^{\mathrm{st}}(\alpha_n)
\]
whose map on geometric points is $(P,Z)\mapsto\iota_{P*}\mathcal I_Z$.

The classification proved above shows that $\Phi$ is bijective on geometric points.
The scheme $H_n$ is projective, while $\M_X^{\mathrm{st}}(\alpha_n)$ is separated, so $\Phi$ is proper.
Its geometric fibers are finite, hence it is quasi-finite and therefore finite.
The $\PP^2$-bundle $\mathcal U\to B$ is Zariski-locally trivial, so $H_n\to B$ is locally a product with $\Hilb^n(\PP^2)$.
Fogarty's theorem therefore shows that $H_n$ is smooth and irreducible of dimension $d:=2n+3$, so the target is irreducible of dimension $d$.
At a point $[F]\in\M_X^{\mathrm{st}}(\alpha_n)$, stable-sheaf deformation theory gives
\[
T_{[F]}\M_X^{\mathrm{st}}(\alpha_n)\cong\Ext^1_X(F,F);
\]
see \cite[Proposition~3.26]{Thomas}.
The Ext calculation in Proposition~\ref{prop:tangent-obstruction-fiber}, which is independent of the present identification, gives
\[
\dim\Ext^1_X(F,F)=2n+3=d.
\]
Thus the local dimension and embedding dimension of $\M_X^{\mathrm{st}}(\alpha_n)$ agree at every closed point.
The target is therefore regular, and in particular normal.
Finally, because the ground field has characteristic zero and $\Phi$ is bijective on geometric points, the finite morphism $\Phi$ is birational.
A finite birational morphism onto a normal scheme is an isomorphism, proving the theorem.
\end{proof}

\begin{corollary}\label{cor:smooth-moduli}
The moduli space $M_n=\Hilb^n(\mathcal U/B)$ is smooth and irreducible of dimension $2n+3$.
\end{corollary}

\begin{proof}
By Theorem~\ref{thm:classification}, the Simpson stable-sheaf moduli scheme is isomorphic to $M_n=\Hilb^n(\mathcal U/B)$, so it suffices to establish the assertions for the relative Hilbert scheme.
The morphism $q\colon \mathcal U\to B$ is a Zariski-locally trivial $\PP^2$-bundle.
Hence, over every trivializing open subset $V\subset B$, one has
\[
\Hilb^n(\mathcal U/B)|_V\cong V\times \Hilb^n(\PP^2).
\]
By Fogarty's theorem, $\Hilb^n(\PP^2)$ is smooth and irreducible of dimension $2n$ \cite{Fogarty}.
Since $B\cong\PP^{3\dual}$ is smooth and irreducible of dimension $3$, it follows that $M_n$ is smooth and irreducible of dimension $2n+3$.
\end{proof}

\begin{remark}\label{rem:joyce}
Because semistability equals stability in the class $\alpha_n$, Joyce's invariant class agrees with the virtual class of the stable-sheaf moduli space; see \cite{JoyceAbelian,Thomas}. This is the relevant wall-crossing framework for $\operatorname{coh}(X)$ when $X$ is a Fano threefold.
\end{remark}

\section{Deformation theory and the obstruction bundle}\label{sec:obstruction}

Let $\rho\colon M_n\to B$ be the support-plane morphism.
Set
\[
\mathcal U_n:=M_n\times_B\mathcal U,
\]
let
\[
\pi\colon \mathcal U_n\to M_n
\]
be the projection, let $j\colon \mathcal U_n\hookrightarrow M_n\times X$ be the natural closed immersion, and let $\mathcal I$ be the universal ideal sheaf of the universal zero-dimensional subscheme in $\mathcal U_n$.
The tautological family of stable sheaves on $M_n\times X$ is
\[
\mathcal F:=j_*\mathcal I.
\]

\subsection{The codimension-one Ext sequence}
We first recall the standard exact sequence for a regular divisor embedding.

\begin{lemma}\label{lem:ext-sequence}
Let $Y$ be a smooth projective threefold, let $i\colon S\hookrightarrow Y$ be a smooth divisor with normal bundle $N_{S/Y}$, and let $E,G$ be coherent sheaves on $S$.
Then there is a functorial exact sequence
\begin{align*}
0 &\to \Ext^1_S(E,G)
\to \Ext^1_Y(i_*E,i_*G)
\to \Hom_S(E,G\otimes N_{S/Y}) \\
&\to \Ext^2_S(E,G)
\to \Ext^2_Y(i_*E,i_*G)
\to \Ext^1_S(E,G\otimes N_{S/Y}) \to 0.
\end{align*}
Moreover, $\Ext^k_S(E,G)=0$ for all $k>2$, and the continuation of the sequence gives a canonical isomorphism
\[
\Ext^3_Y(i_*E,i_*G)\cong \Ext^2_S(E,G\otimes N_{S/Y}).
\]
There is also an Euler characteristic identity
\begin{equation}\label{eq:chi-general}
\chi_Y(i_*E,i_*G)=\chi_S(E,G)-\chi_S(E,G\otimes N_{S/Y}).
\end{equation}
\end{lemma}

\begin{proof}
Since $S\subset Y$ is a smooth divisor, the inclusion $i\colon S\hookrightarrow Y$ is a regular closed immersion of codimension one, and $i_*\OO_S$ has the Koszul resolution
\[
0\longrightarrow \OO_Y(-S)\longrightarrow \OO_Y
\longrightarrow i_*\OO_S\longrightarrow 0,
\]
where $\OO_Y(-S)|_S\cong N_{S/Y}^{\dual}$.
The corresponding derived pullback calculation gives
\[
\mathcal H^0(Li^*i_*E)\cong E,
\qquad
\mathcal H^{-1}(Li^*i_*E)\cong E\otimes N_{S/Y}^{\dual},
\]
and no other cohomology sheaves.
The canonical truncation triangle is therefore
\[
(E\otimes N_{S/Y}^{\dual})[1]
\longrightarrow Li^*i_*E
\longrightarrow E
\longrightarrow (E\otimes N_{S/Y}^{\dual})[2].
\]
Applying $R\!\Hom_S(-,G)$ and using derived adjunction gives the distinguished triangle
\begin{align*}
R\!\Hom_S(E,G)
&\longrightarrow R\!\Hom_Y(i_*E,i_*G) \\
&\longrightarrow R\!\Hom_S(E,G\otimes N_{S/Y})[-1]
\longrightarrow R\!\Hom_S(E,G)[1].
\end{align*}
Equivalently, this gives the regular-immersion spectral sequence
\[
E_2^{p,q}=\Ext^p_S\bigl(E,G\otimes \textstyle\bigwedge^q N_{S/Y}\bigr)
\Longrightarrow \Ext^{p+q}_Y(i_*E,i_*G),
\qquad q=0,1;
\]
see also \cite[p.~1055]{GSlinear}.
The low-degree portion obtained by taking cohomology is the displayed exact sequence.
The normal bundle occurs because the derived pullback contains the conormal factor $N_{S/Y}^{\dual}$, which becomes $N_{S/Y}$ after applying $R\!\Hom_S(-,G)$.

Since $S$ is a smooth projective surface, derived Serre duality gives
\[
\Ext^k_S(E,G)
\cong \Ext^{2-k}_S(G,E\otimes K_S)^{\dual}=0
\qquad\text{for all }k>2.
\]
Thus the displayed sequence ends in zero, and its continuation gives the stated isomorphism in degree three.
Finally, taking alternating sums in the distinguished triangle gives the Euler characteristic identity; the shift by $[-1]$ accounts for the minus sign.
\end{proof}

For a plane $P\subset X=\PP^3$, write $i\colon P\hookrightarrow X$ for the inclusion.
The normal bundle is
\[
N_{P/X}\cong \OO_P(1).
\]
Consequently, the beginning of the exact sequence is
\[
0\longrightarrow \Ext^1_P(E,G)
\longrightarrow \Ext^1_X(i_*E,i_*G)
\longrightarrow \Hom_P(E,G(1))
\longrightarrow \Ext^2_P(E,G).
\]
When $E=G$, the first term describes deformations with the support plane fixed, whereas the Hom term records the component of an ambient deformation normal to $P$.
The specialization of Lemma~\ref{lem:ext-sequence} to $E=G=\I_Z$ will be applied in Proposition~\ref{prop:tangent-obstruction-fiber}.
We first record the required surface calculations.

\begin{lemma}\label{lem:surface-ext-vanishing}
Let $P\cong \PP^2$ and let $E=\I_Z$ for $Z\in P^{[n]}$.
Then
\[
\Ext^2_P(E,E)=0,
\qquad
\Ext^2_P(E,E(1))=0,
\qquad
\Hom_P(E,E(1))\cong H^0\bigl(\OO_P(1)\bigr).
\]
Moreover,
\[
\chi_P(E,E)=1-2n,
\qquad
\chi_P(E,E(1))=3-2n.
\]
\end{lemma}

\begin{proof}
By Serre duality,
\[
\Ext^2_P(E,E)\cong \Hom_P(E,E\otimes K_P)^\dual=\Hom_P(E,E(-3))^\dual,
\]
and similarly
\[
\Ext^2_P(E,E(1))\cong \Hom_P(E,E(-2))^\dual.
\]
Since $E$ is Gieseker stable by Lemma~\ref{lem:rank-one-stable}, while $E(-3)$ and $E(-2)$ are stable with strictly smaller reduced Hilbert polynomial, both Hom groups vanish.
Hence the two $\Ext^2$ groups vanish.

Because $E$ is rank one and torsion-free on the smooth surface $P$, we have $\mathcal Hom_P(E,E)\cong \OO_P$.
Tensoring by $\OO_P(1)$ gives $\mathcal Hom_P(E,E(1))\cong \OO_P(1)$, and therefore
\[
\Hom_P(E,E(1))\cong H^0\bigl(\OO_P(1)\bigr).
\]

Finally, the Riemann--Roch formula on $P\cong \PP^2$ with $\ch(E)=1-n[\mathrm{pt}]$ gives
\[
\chi_P(E,E)=\int_P (1-2n[\mathrm{pt}])\,\td(P)=1-2n,
\]
and
\[
\chi_P(E,E(1))=\int_P (1-2n[\mathrm{pt}])\,\ch\bigl(\OO_P(1)\bigr)\,\td(P)=3-2n.
\]
\end{proof}

\begin{proposition}\label{prop:tangent-obstruction-fiber}
Let $P\subset X$ be a plane, let $Z\in P^{[n]}$, and set $F=\iota_{P*}\I_Z$.
Then there are canonical exact sequences
\begin{equation}\label{eq:tangent-exact-fiber}
0\longrightarrow \Ext^1_P(\I_Z,\I_Z)
\longrightarrow \Ext^1_X(F,F)
\longrightarrow H^0\bigl(\OO_P(1)\bigr)
\longrightarrow 0
\end{equation}
and
\begin{equation}\label{eq:obstruction-fiber}
\Ext^2_X(F,F)\cong \Ext^1_P\bigl(\I_Z,\I_Z(1)\bigr).
\end{equation}
In particular,
\[
\dim \Ext^1_X(F,F)=2n+3,
\qquad
\dim \Ext^2_X(F,F)=2n.
\]
Therefore $\vdim M_n=3$.
\end{proposition}

\begin{proof}
Since $P\subset X=\PP^3$ is a hyperplane divisor, its normal bundle is
\[
N_{P/X}=N_{P/\PP^3}\cong \OO_X(P)|_P\cong\OO_P(1).
\]
Apply Lemma~\ref{lem:ext-sequence} with $Y=X=\PP^3$, $S=P$, and $E=G=\I_Z$.
By Lemma~\ref{lem:surface-ext-vanishing},
\[
\Ext^2_P(\I_Z,\I_Z)=0,
\qquad
\Hom_P(\I_Z,\I_Z(1))\cong H^0\bigl(\OO_P(1)\bigr).
\]
Moreover, $\Ext^3_P(\I_Z,\I_Z)=0$ because $P$ is a smooth projective surface, as recorded in Lemma~\ref{lem:ext-sequence}.
The exact sequence in that lemma therefore gives \eqref{eq:tangent-exact-fiber} and the isomorphism \eqref{eq:obstruction-fiber}.
The same lemma also gives $\chi_P(\I_Z,\I_Z)=1-2n$, while $\Hom_P(\I_Z,\I_Z)\cong\C$.
Together with the vanishing of $\Ext^2_P(\I_Z,\I_Z)$, this implies $\dim\Ext^1_P(\I_Z,\I_Z)=2n$.
This gives $\dim\Ext^1_X(F,F)=2n+3$.

The Euler characteristic identity \eqref{eq:chi-general} and Lemma~\ref{lem:surface-ext-vanishing} imply
\[
\chi_X(F,F)=\chi_P(\I_Z,\I_Z)-\chi_P\bigl(\I_Z,\I_Z(1)\bigr)=-2.
\]
Moreover $F$ is stable, so $\Hom_X(F,F)=\C$.
Since $K_X\cong\OO_X(-4)$, Serre duality on $X=\PP^3$, the projection formula, and the full faithfulness of pushforward along $\iota_P$ give
\[
\begin{aligned}
\Ext^3_X(F,F)
&\cong \Hom_X(F,F\otimes K_X)^\dual \\
&\cong \Hom_X\bigl(\iota_{P*}\I_Z,\iota_{P*}(\I_Z(-4))\bigr)^\dual \\
&\cong \Hom_P\bigl(\I_Z,\I_Z(-4)\bigr)^\dual \\
&\cong H^0\bigl(P,\OO_P(-4)\bigr)^\dual=0.
\end{aligned}
\]
For the last isomorphism, we use $\mathcal Hom_P(\I_Z,\I_Z)\cong\OO_P$.
Hence
\[
\dim \Ext^2_X(F,F)=\dim\Ext^1_X(F,F)-3=2n.
\]
This proves the virtual dimension statement.
\end{proof}

\subsection{Global identification of the obstruction bundle}
We now globalize the previous fiberwise computation.
The universal plane is the zero locus of the tautological section of
\[
\OO_B(1)\boxtimes \OO_X(1)
\]
on $B\times X$; therefore
\begin{equation}\label{eq:normal-bundle}
N_{\mathcal U/(B\times X)}\cong \OO_{\mathcal U/B}(1)\otimes q^*\OO_B(1).
\end{equation}

\begin{definition}\label{def:twisted-tangent-relative}
Let $S$ be a smooth projective surface and let $L$ be a line bundle on $S$.
If $\mathcal J$ denotes the universal ideal sheaf on $S\times S^{[n]}$ and $p\colon S\times S^{[n]}\to S^{[n]}$ is the projection, the Carlsson--Okounkov twisted tangent class is
\begin{equation}\label{eq:def-twisted-tangent}
\Ttw_S^{[n]}(L):= \chi(S,L)[\mathcal{O}_{S^{[n]}}] - Rp_*R\mathcal Hom(\mathcal J,\mathcal J\otimes L)\in K^0(S^{[n]}).
\end{equation}

For the relative family $q\colon \mathcal U\to B$, recall that $\rho\colon M_n=\Hilb^n(\mathcal U/B)\to B$ is the structure morphism and set $\mathcal U_n=M_n\times_B\mathcal U$.
The projections fit into the Cartesian square
\[
\begin{CD}
\mathcal U_n @>{r}>> \mathcal U \\
@V{\pi}VV @VV{q}V \\
M_n @>{\rho}>> B.
\end{CD}
\]
Let $\mathcal I$ be the universal ideal sheaf on $\mathcal U_n$.
We define the relative twisted tangent class by
\begin{equation}\label{eq:def-relative-twisted}
\begin{aligned}
\Ttw_{\mathcal U/B}^{[n]}\bigl(\OO_{\mathcal U/B}(1)\bigr)
:={}&\rho^*Rq_*\OO_{\mathcal U/B}(1) \\
&\quad-R\pi_*R\mathcal Hom\bigl(\mathcal I,\mathcal I\otimes r^*\OO_{\mathcal U/B}(1)\bigr)
\in K^0(M_n).
\end{aligned}
\end{equation}
By flat base change, there is a canonical identification
\[
\rho^*Rq_*\OO_{\mathcal U/B}(1)
\cong R\pi_*r^*\OO_{\mathcal U/B}(1).
\]
\end{definition}

\begin{theorem}\label{thm:obstruction-bundle}
The obstruction sheaf of the standard truncated perfect obstruction theory on $M_n$ for rank-zero stable sheaves on $X$ is the vector bundle
\begin{equation}\label{eq:ob-bundle}
\mathrm{Ob}_n\cong
\mathcal Ext^1_{\pi}\bigl(\mathcal I,\mathcal I\otimes r^*N_{\mathcal U/(B\times X)}\bigr)
\cong
\Ttw_{\mathcal U/B}^{[n]}\bigl(\OO_{\mathcal U/B}(1)\bigr)\otimes \rho^*\OO_B(1).
\end{equation}
It has rank $2n$, and the virtual class is
\begin{equation}\label{eq:virtual-class-top-chern}
\vfc{M_n}=c_{2n}(\mathrm{Ob}_n)\cap [M_n].
\end{equation}
\end{theorem}

\begin{proof}
By Theorem~\ref{thm:classification}, the scheme $M_n$ is the stable-sheaf moduli scheme.
Recall that
\[
j\colon \mathcal U_n=M_n\times_B\mathcal U\hookrightarrow M_n\times X
\]
is the natural closed immersion and that $\mathcal I$ is the universal ideal sheaf on $\mathcal U_n$.
The tautological family is
\[
\mathcal F:=j_*\mathcal I.
\]
It is flat over $M_n$, and its fiber at $m=(P,Z)$ is
\[
\mathcal F|_{\{m\}\times X}\cong \iota_{P*}\I_Z.
\]

Let $\mathbb{L}_{M_n}:=\mathbb{L}_{M_n/\C}$ denote the absolute cotangent complex, and set
\[
L_{M_n}^\bullet:=\tau^{[-1,0]}\mathbb{L}_{M_n}.
\]
A perfect obstruction theory is a morphism $\phi\colon E^\bullet\to L_{M_n}^\bullet$, with $E^\bullet$ perfect of amplitude $[-1,0]$, such that $h^0(\phi)$ is an isomorphism and $h^{-1}(\phi)$ is surjective.

Since the sheaves parametrized by $M_n$ have rank zero, the trace map does not split off their scalar endomorphisms: indeed, $\operatorname{tr}(\operatorname{id}_F)=\rk(F)=0$.
On the other hand,
\[
H^1(X,\OO_X)=H^2(X,\OO_X)=0,
\]
so the trace-free and full deformation and obstruction groups agree in degrees $1$ and $2$.
Moreover, Proposition~\ref{prop:tangent-obstruction-fiber} gives $\Ext^3_X(F,F)=0$.
In particular, Thomas' degree-three trace-free vanishing hypothesis is satisfied.
We therefore remove the degree-zero scalar term by truncating the derived endomorphism complex to degrees $1$ and $2$.
Let $p_M\colon M_n\times X\to M_n$ be the projection.
The component of the Atiyah class of $\mathcal F$ in the $M_n$-direction induces, by Thomas' construction \cite{Thomas} in the truncated rank-zero form of \cite[Theorem~2.3]{GSmod}, the perfect obstruction theory
\[
E^\bullet:=
\left(\tau^{[1,2]}R p_{M*}R\mathcal Hom(\mathcal F,\mathcal F)\right)^\dual[-1]
\xrightarrow{\ \phi\ } L_{M_n}^\bullet.
\]
Here $\tau^{[1,2]}$ retains precisely the cohomology sheaves in degrees $1$ and $2$, so $E^\bullet$ has amplitude $[-1,0]$.
At a closed point, the corresponding tangent and obstruction spaces are $\Ext^1_X(F,F)$ and $\Ext^2_X(F,F)$, respectively.
Since $M_n$ is smooth, $L_{M_n}^\bullet\simeq\Omega_{M_n}[0]$; nevertheless, $h^1((E^\bullet)^\dual)$ may be nonzero and is the obstruction sheaf of this perfect obstruction theory.

Apply Lemma~\ref{lem:ext-sequence} fiberwise to the family
\[
j\colon \mathcal U_n\hookrightarrow M_n\times X.
\]
The relative form of that sequence, together with cohomology and base change, gives the short exact sequence of vector bundles
\begin{equation}\label{eq:global-tangent-exact}
0\longrightarrow T_{M_n/B}\longrightarrow T_{M_n}\longrightarrow \rho^*T_B\longrightarrow 0
\end{equation}
and an isomorphism of obstruction sheaves
\[
\mathcal Ext^2_{p_M}(\mathcal F,\mathcal F)
\cong
\mathcal Ext^1_{\pi}\bigl(\mathcal I,\mathcal I\otimes r^*N_{\mathcal U/(B\times X)}\bigr).
\]
This is the first isomorphism in \eqref{eq:ob-bundle}.

Using \eqref{eq:normal-bundle} and the identity
\[
r^*q^*\OO_B(1)=\pi^*\rho^*\OO_B(1),
\]
we may factor the twist as
\[
\mathcal Ext^1_{\pi}\bigl(\mathcal I,\mathcal I\otimes r^*N_{\mathcal U/(B\times X)}\bigr)
\cong
\mathcal Ext^1_{\pi}\bigl(\mathcal I,\mathcal I\otimes r^*\OO_{\mathcal U/B}(1)\bigr)\otimes \rho^*\OO_B(1).
\]
Set
\[
\mathcal C:=
R\pi_*R\mathcal Hom\bigl(\mathcal I,
\mathcal I\otimes r^*\OO_{\mathcal U/B}(1)\bigr).
\]
The universal zero-dimensional subscheme is finite and flat over $M_n$.
It follows from its universal ideal sequence that $\mathcal I$ is $\pi$-flat; it is also $\pi$-perfect.
Since $\pi$ is smooth and projective of relative dimension two, $\mathcal C$ is a perfect complex on $M_n$, and its formation commutes with derived base change.
Thus, at $m=(P,Z)$,
\[
\mathcal C\otimes^{\mathbf L}k(m)
\simeq
R\Hom_P\bigl(\I_Z,\I_Z(1)\bigr).
\]
By Lemma~\ref{lem:surface-ext-vanishing},
\[
\Ext^2_P\bigl(\I_Z,\I_Z(1)\bigr)=0,
\qquad
\Hom_P\bigl(\I_Z,\I_Z(1)\bigr)
\cong H^0\bigl(P,\OO_P(1)\bigr),
\]
and
\[
\chi_P\bigl(\I_Z,\I_Z(1)\bigr)=3-2n.
\]
The groups in degrees greater than two vanish because $P$ is a smooth surface, and there are no negative Ext groups.
Consequently the fiber cohomology of $\mathcal C$ is concentrated in degrees zero and one, with constant dimensions
\[
\dim\Hom_P\bigl(\I_Z,\I_Z(1)\bigr)=3,
\qquad
\dim\Ext^1_P\bigl(\I_Z,\I_Z(1)\bigr)=2n.
\]
Cohomology and base change for perfect complexes therefore shows that
\[
\mathcal H^0(\mathcal C)
=\mathcal Hom_{\pi}\bigl(\mathcal I,
\mathcal I\otimes r^*\OO_{\mathcal U/B}(1)\bigr)
\]
and
\[
\mathcal H^1(\mathcal C)
=\mathcal Ext^1_{\pi}\bigl(\mathcal I,
\mathcal I\otimes r^*\OO_{\mathcal U/B}(1)\bigr)
\]
are locally free of ranks $3$ and $2n$, respectively, their formation commutes with base change, and all other cohomology sheaves of $\mathcal C$ vanish.
In particular, the second sheaf is an honest vector bundle of rank $2n$.

Multiplication by sections induces a morphism
\[
\rho^*q_*\OO_{\mathcal U/B}(1)
\cong \pi_*r^*\OO_{\mathcal U/B}(1)
\longrightarrow
\mathcal Hom_{\pi}\bigl(\mathcal I,
\mathcal I\otimes r^*\OO_{\mathcal U/B}(1)\bigr).
\]
Both sides are locally free of rank $3$, and the morphism is an isomorphism on every fiber by the preceding identification of Hom groups; hence it is an isomorphism.
Since $q$ is a $\PP^2$-bundle, its higher direct images of $\OO_{\mathcal U/B}(1)$ vanish.
Thus the nonzero degree-zero term $\mathcal H^0(\mathcal C)$ cancels canonically with the first term in Definition~\ref{def:twisted-tangent-relative}, and the resulting relative $K$-class is represented by the vector bundle
\[
\mathcal Ext^1_{\pi}\bigl(\mathcal I,\mathcal I\otimes r^*\OO_{\mathcal U/B}(1)\bigr).
\]
This proves the second isomorphism in \eqref{eq:ob-bundle}.

Finally, $M_n$ is smooth and the obstruction sheaf is a vector bundle.
For a perfect obstruction theory on a smooth scheme, the virtual class is the top Chern class of the obstruction bundle; see Behrend--Fantechi \cite{BF}.
Thus \eqref{eq:virtual-class-top-chern} holds.
\end{proof}

\begin{remark}\label{rem:tbundle}
The exact sequence \eqref{eq:global-tangent-exact} reflects the decomposition of the tangent space into deformations of the zero-dimensional subscheme inside a fixed plane and deformations of the support plane itself.
Fiberwise the quotient is
\[
H^0\bigl(\OO_P(1)\bigr)\cong T_{B,[P]}.
\]
\end{remark}

\section{The planar Donaldson--Thomas series}\label{sec:series}

Let $h=c_1(\OO_B(1))$ and define $\Gamma_n$ by \eqref{eq:intro-gamma}.
Because $\rho_*\vfc{M_n}\in A_3(B)$ and $A_3(B)\cong \Z[B]$, there is a unique scalar $\gamma_n$ such that
\[
\rho_*\vfc{M_n}=\gamma_n[B].
\]
Intersecting with $h^3$ shows that $\gamma_n=\Gamma_n$.
Thus the generating series of the numerical invariants and the homology-valued generating series contain the same information.

\begin{proposition}\label{prop:fiber-integral}
For every $n\ge 0$,
\begin{equation}\label{eq:gamma-fiber-integral}
\Gamma_n=\int_{(\PP^2)^{[n]}} c_{2n}\Bigl(\Ttw_{\PP^2}^{[n]}\bigl(\OO_{\PP^2}(1)\bigr)\Bigr).
\end{equation}
Equivalently,
\[
\rho_*\vfc{M_n}=\Gamma_n[B].
\]
\end{proposition}

\begin{proof}
By Theorem~\ref{thm:obstruction-bundle},
\[
\vfc{M_n}=c_{2n}(\mathrm{Ob}_n)\cap [M_n].
\]
The fiber of $\rho$ over a point $[P]\in B$ is canonically $P^{[n]}\cong (\PP^2)^{[n]}$.
Since $\rho^*\OO_B(1)$ restricts trivially to the fiber, the restriction of the obstruction bundle is
\[
\mathrm{Ob}_n|_{\rho^{-1}([P])}
\cong
\Ttw_{P}^{[n]}\bigl(\OO_P(1)\bigr)
\cong
\Ttw_{\PP^2}^{[n]}\bigl(\OO_{\PP^2}(1)\bigr).
\]
Therefore the coefficient of $[B]$ in $\rho_*\vfc{M_n}$ is exactly the fiber integral in \eqref{eq:gamma-fiber-integral}.
\end{proof}

The closed form now follows from the Carlsson--Okounkov product formula.

\begin{theorem}\label{thm:planar-series}
The planar series is
\[
\sum_{n\ge 0}\Gamma_nq^n=\prod_{m\ge 1}(1-q^m)^{-7}.
\]
Equivalently,
\[
\sum_{n\ge 0}\rho_*\vfc{M_n}\,q^n
=
\left(\prod_{m\ge 1}(1-q^m)^{-7}\right)[B].
\]
\end{theorem}

\begin{proof}
Proposition~\ref{prop:fiber-integral} reduces the problem to the surface integral
\[
\int_{(\PP^2)^{[n]}} c_{2n}\Bigl(\Ttw_{\PP^2}^{[n]}\bigl(\OO_{\PP^2}(1)\bigr)\Bigr).
\]
By Carlsson--Okounkov's formula \cite[Cor.~1]{CO}, for any smooth projective surface $S$ and line bundle $L$ one has
\[
\sum_{n\ge 0} q^n\int_{S^{[n]}} e\bigl(\Ttw_S^{[n]}(L)\bigr)
=
\prod_{m\ge 1}(1-q^m)^{(L,K_S-L)-e(S)}.
\]
Applying this with $S=\PP^2$ and $L=\OO_{\PP^2}(1)$ gives
\[
(L,K_S-L)-e(S)=H\cdot (-4H)-3=-7.
\]
Since $\Ttw_{\PP^2}^{[n]}(\OO(1))$ is an honest rank-$2n$ vector bundle, its Euler class is its top Chern class, and the formula follows.
\end{proof}

\begin{remark}\label{rem:eta}
Writing $q=e^{2\pi i\tau}$, the product in Theorem~\ref{thm:planar-series} is $q^{7/24}\eta(\tau)^{-7}$.
The first coefficients are
\[
1,\ 7,\ 35,\ 140,\ 490,\ 1547,\ldots.
\]
These numbers have a direct enumerative interpretation.
The insertion $\rho^*(h^3)$ fixes the support plane $P\subset \PP^3$, so $\Gamma_n$ is the degree of the resulting zero-dimensional virtual class of sheaves $i_{P*}\I_Z$ with $Z\subset P$ of length $n$.
Equivalently,
\[
\Gamma_n
=
\int_{P^{[n]}}c_{2n}\Bigl(\Ttw_P^{[n]}\bigl(\OO_P(1)\bigr)\Bigr),
\]
the Euler number of the rank-$2n$ obstruction bundle on $P^{[n]}$.
Thus $\Gamma_n$ is a virtual count, rather than the cardinality of a finite set of subschemes.

Combinatorially, $\Gamma_n$ is the number $p_7(n)$ of seven-colored partitions of $n$, or equivalently of ordered seven-tuples of ordinary partitions whose total size is $n$.
For example, $p_7(1)=7$ and
\[
p_7(2)=7+\binom{8}{2}=35.
\]
The exponent $7$ is itself the geometric characteristic number
\[
e(\PP^2)+H\cdot\bigl(H-K_{\PP^2}\bigr)
=3+4
=\int_{\PP^2}c_2\bigl(T_{\PP^2}\otimes\OO_{\PP^2}(1)\bigr).
\]
By G\"ottsche's formula \cite{Gottsche}, the same coefficients are also the Euler characteristics $e(S^{[n]})$ for any smooth projective surface $S$ with $e(S)=7$; they are not the ordinary Euler characteristics of $(\PP^2)^{[n]}$, whose generating series has exponent $3$.
Thus the natural point-inserted two-dimensional Donaldson--Thomas series on $\PP^3$ is modular, in sharp contrast with the degree-zero point series $M(-q)^{-20}$.
\end{remark}

\begin{remark}\label{rem:linear-systems}
Theorem~\ref{thm:planar-series} is the simplest moving-divisor instance of the linear-system framework studied by Gholampour and Sheshmani \cite{GSlinear}.
In the planar case the base of the support family is $B=(\PP^3)^\dual$, and the point insertion $h^3$ exactly saturates the virtual dimension.
\end{remark}

\section{The planar locus inside $\Hilb^n(\PP^3)$}\label{sec:locus}

Let $X^{[n]}:=\Hilb^n(\PP^3)$, and let $X^{[n]}_{\mathrm{dist}}\subset X^{[n]}$ be the open subscheme parametrizing reduced length-$n$ subschemes, or equivalently unordered collections of $n$ distinct points.
If
\[
\Delta:=\bigcup_{1\leq i<j\leq n}\Delta_{ij}\subset X^n
\]
denotes the big diagonal, then there is a natural isomorphism
\begin{equation}\label{eq:distinct-point-configuration}
X^{[n]}_{\mathrm{dist}}
\cong
\bigl(X^n\setminus\Delta\bigr)/\mathfrak S_n,
\end{equation}
where $\mathfrak S_n$ acts by permuting the ordered points; compare \cite[\S~1.1]{Nakajima}.
Since $X=\PP^3$ is smooth and irreducible and the action on $X^n\setminus\Delta$ is free, \eqref{eq:distinct-point-configuration} shows that $X^{[n]}_{\mathrm{dist}}$ is smooth and irreducible of dimension $3n$.

We define the \emph{principal component}, also called the \emph{smoothable component}, by
\begin{equation}\label{eq:principal-component}
X^{[n]}_{\mathrm{prin}}
:=
\left(\overline{X^{[n]}_{\mathrm{dist}}}\right)_{\mathrm{red}}
\subset X^{[n]},
\end{equation}
where the bar denotes closure in $X^{[n]}$ and the closure is endowed with its induced reduced closed-subscheme structure.
At a point $[Z]\in X^{[n]}_{\mathrm{dist}}$, with $Z=\{x_1,\ldots,x_n\}$, one has
\[
T_{[Z]}X^{[n]}
\cong
\Hom_X(\I_Z,\OO_Z)
\cong
\bigoplus_{i=1}^n T_{x_i}X.
\]
Thus every point of $X^{[n]}_{\mathrm{dist}}$ is a smooth point of $X^{[n]}$ of local dimension $3n$, and the scheme $X^{[n]}_{\mathrm{prin}}$ defined in \eqref{eq:principal-component} is the unique irreducible component of $X^{[n]}$ containing the distinct-point locus.  In particular,
\[
\dim X^{[n]}_{\mathrm{prin}}=3n.
\]
Its closed points are precisely the length-$n$ subschemes which occur as flat limits in $X$ of $n$ distinct points.  The component itself need not be smooth.

For the moment, let
\[
X^{[n]}_{\mathrm{pl}}(\C)
:=
\left\{
[Z]\in X^{[n]}(\C):
\begin{array}{l}
Z\subset P \text{ scheme-theoretically}\\
\text{for some plane }P\subset \PP^3
\end{array}
\right\}
\]
be the set of planar closed points.
Recall from Section~\ref{sec:obstruction} that $\mathcal I$ is the ideal sheaf of the universal length-$n$ subscheme
\[
\mathcal Z_{M_n}\subset \mathcal U_n.
\]
Via the closed immersion $j\colon \mathcal U_n\hookrightarrow M_n\times X$, we regard $\mathcal Z_{M_n}$ as a closed subscheme of $M_n\times X$.
It is finite and flat of degree $n$ over $M_n$.
The universal property of the Hilbert scheme therefore gives a morphism
\begin{equation}\label{eq:incidence-morphism}
r_n\colon M_n\longrightarrow X^{[n]},\qquad r_n(P,Z)=[Z].
\end{equation}
The morphism $r_n$ in \eqref{eq:incidence-morphism} is induced by the universal family; the equality $r_n(P,Z)=[Z]$ records its action on geometric points.
Its image on geometric points is precisely $X^{[n]}_{\mathrm{pl}}(\C)$.

Let
\[
\mathcal Z\subset X^{[n]}\times X
\]
be the universal zero-dimensional subscheme.
Let $\operatorname{pr}_X\colon X^{[n]}\times X\to X$ be the second projection, and write
\[
p\colon \mathcal Z\longrightarrow X^{[n]},
\qquad
q:=\operatorname{pr}_X|_{\mathcal Z}\colon \mathcal Z\longrightarrow X
\]
for the induced projections.
Since $p$ is finite and flat of degree $n$ and $q^*\OO_X(1)$ is invertible on $\mathcal Z$, the tautological sheaf
\[
\OO_X(1)^{[n]}:=p_*q^*\OO_X(1)
\]
is locally free of rank $n$ on $X^{[n]}$.
Indeed, for every geometric point $[Z]\in X^{[n]}$, finite flat base change identifies its fiber with
\[
\OO_X(1)^{[n]}|_{[Z]}
\cong H^0\bigl(Z,\OO_Z(1)\bigr),
\]
which has dimension $n$ because $Z$ has length $n$ and $\OO_Z(1)$ is invertible.

The usual evaluation morphism
\[
H^0\bigl(X,\OO_X(1)\bigr)\otimes\OO_X\longrightarrow\OO_X(1)
\]
pulls back along $q$ to a morphism on $\mathcal Z$.
By the adjunction $p^*\dashv p_*$, it induces the bundle map
\begin{equation}\label{eq:eval-map}
\mathrm{ev}\colon H^0\bigl(\OO_X(1)\bigr)\otimes \OO_{X^{[n]}}\longrightarrow \OO_X(1)^{[n]}.
\end{equation}
Its fiber at $[Z]$ is the restriction map
\[
H^0\bigl(X,\OO_X(1)\bigr)\longrightarrow H^0\bigl(Z,\OO_Z(1)\bigr).
\]
Let $D_3(\mathrm{ev})\subset X^{[n]}$ denote the rank-$\leq 3$ degeneracy subscheme of \eqref{eq:eval-map}; it is locally defined by the $4\times4$ minors of a matrix representing $\mathrm{ev}$.

\begin{proposition}\label{prop:planar-locus-degeneracy}
The closed points of $D_3(\mathrm{ev})$ are precisely the planar points $X^{[n]}_{\mathrm{pl}}(\C)$.
In particular, the planar set is Zariski closed.
Endow the corresponding closed subset with its induced reduced closed-subscheme structure and denote the resulting scheme by $X^{[n]}_{\mathrm{pl}}$.
Then
\begin{equation}\label{eq:planar-locus-reduced-degeneracy}
X^{[n]}_{\mathrm{pl}}=\bigl(D_3(\mathrm{ev})\bigr)_{\mathrm{red}}.
\end{equation}
Equivalently, a length-$n$ subscheme $Z\subset \PP^3$ is planar if and only if the restriction map
\[
H^0\bigl(X,\OO_X(1)\bigr)\longrightarrow H^0\bigl(Z,\OO_Z(1)\bigr)
\]
has nonzero kernel.
For $n\leq 3$, every length-$n$ subscheme is planar, and
\[
X^{[n]}_{\mathrm{pl}}=X^{[n]}.
\]
For every $n$, the morphism $r_n$ in \eqref{eq:incidence-morphism} factors through $X^{[n]}_{\mathrm{pl}}$.
For $n\geq 4$, the planar locus is irreducible of dimension $2n+3$ and has codimension $n-3$ inside $X^{[n]}_{\mathrm{prin}}$.
For $n\geq 4$, moreover, the induced morphism
\[
r_n\colon M_n\longrightarrow X^{[n]}_{\mathrm{pl}}
\]
is birational and is an isomorphism over the open locus of subschemes spanning a unique plane.
\end{proposition}

\begin{proof}
Let $[Z]\in X^{[n]}$ be a geometric point, and let $\I_Z\subset\OO_X$ be its ideal sheaf.
Twisting the ideal-sheaf sequence of $Z$ by $\OO_X(1)$ gives
\begin{equation}\label{eq:planar-ideal-sequence}
0\longrightarrow \I_Z(1)
\longrightarrow \OO_X(1)
\longrightarrow \OO_Z(1)
\longrightarrow 0.
\end{equation}
Taking global sections in \eqref{eq:planar-ideal-sequence} and using the fiber description of \eqref{eq:eval-map} gives
\[
\ker\bigl(\mathrm{ev}_{[Z]}\bigr)=H^0\bigl(X,\I_Z(1)\bigr).
\]
A nonzero element $\ell\in H^0(X,\I_Z(1))$ is a linear form which vanishes on $Z$ scheme-theoretically.
It defines a plane $P_\ell:=V(\ell)\subset X$, and the vanishing condition is equivalent to
\[
\I_{P_\ell}\subseteq \I_Z,
\]
which is precisely the scheme-theoretic containment $Z\subseteq P_\ell$.
Conversely, if $Z\subseteq P$ for a plane $P=V(\ell)$, then $\I_P\subseteq\I_Z$, and therefore $0\neq\ell\in H^0(X,\I_Z(1))$.
Consequently,
\[
Z\text{ is planar}
\quad\Longleftrightarrow\quad
\ker\bigl(\mathrm{ev}_{[Z]}\bigr)\neq 0.
\]

Since $\dim H^0(X,\OO_X(1))=4$, this is equivalent to
\[
\operatorname{rank}\bigl(\mathrm{ev}_{[Z]}\bigr)\leq 3.
\]
By the definition of the degeneracy locus, its closed points are therefore precisely $X^{[n]}_{\mathrm{pl}}(\C)$.
Taking the induced reduced closed-subscheme structures proves \eqref{eq:planar-locus-reduced-degeneracy}.

Suppose that $n\leq 3$.
For a reduced subscheme consisting of $n$ distinct points, the elementary geometric reason for planarity is that its linear span has projective dimension at most $n-1\leq 2$, and is therefore contained in a plane in $\PP^3$.
The same conclusion holds for an arbitrary, possibly nonreduced, length-$n$ subscheme.
Indeed,
\[
\dim H^0\bigl(Z,\OO_Z(1)\bigr)=n,
\]
whereas $\dim H^0(X,\OO_X(1))=4$.
Consequently,
\[
\dim\ker\bigl(\mathrm{ev}_{[Z]}\bigr)\geq 4-n\geq 1,
\]
and the preceding argument shows that $Z$ is contained scheme-theoretically in the plane defined by any nonzero element of this kernel.

Scheme-theoretically, the target of \eqref{eq:eval-map} has rank $n\leq3$, so the rank-$\leq3$ condition is automatic and
\[
D_3(\mathrm{ev})=X^{[n]}.
\]
Moreover, the Hilbert scheme of at most three points on a smooth variety is smooth \cite[\S~7.2]{FGlocal}; in particular, $X^{[n]}$ is reduced.
It follows from \eqref{eq:planar-locus-reduced-degeneracy} that
\[
X^{[n]}_{\mathrm{pl}}
=\bigl(D_3(\mathrm{ev})\bigr)_{\mathrm{red}}
=X^{[n]}.
\]

The set-theoretic image of $r_n$ is the planar subset.
Since $M_n$ is smooth, and hence reduced, the morphism $r_n$ factors through the reduced closed subscheme $X^{[n]}_{\mathrm{pl}}$.
Every planar subscheme admits a containing plane, so the induced morphism is surjective on geometric points.
Its image is therefore dense, and since $M_n$ is irreducible, it follows that $X^{[n]}_{\mathrm{pl}}$ is irreducible.

Assume now that $n\geq 4$, and let $U\subset X^{[n]}_{\mathrm{pl}}$ be the open locus on which $\mathrm{ev}$ has rank exactly $3$.
This locus is nonempty: one may take $n$ distinct points on a plane which are not all contained in a line.
On $U$, the kernel of $\mathrm{ev}$ is a line subbundle of
\[
H^0\bigl(X,\OO_X(1)\bigr)\otimes\OO_U.
\]
This line subbundle determines algebraically the unique plane containing the universal subscheme over $U$.
It therefore defines an inverse to $r_n$ over $U$, so
\[
r_n^{-1}(U)\xrightarrow{\ \cong\ }U.
\]
Thus $r_n$ is birational, and
\[
\dim X^{[n]}_{\mathrm{pl}}=\dim M_n=2n+3.
\]

Finally, every planar length-$n$ subscheme $Z\subset P$ is smoothable inside $P$.
Indeed, by Fogarty's theorem, $P^{[n]}$ is irreducible and its open locus of $n$ distinct points is dense.
Thus every point of $X^{[n]}_{\mathrm{pl}}$ belongs to $X^{[n]}_{\mathrm{prin}}$.
Since $X^{[n]}_{\mathrm{prin}}$ has dimension $3n$, the codimension of the planar locus is
\[
3n-(2n+3)=n-3.
\]
\end{proof}

\begin{remark}\label{rem:jrs}
The global geometry of $\Hilb^n(\PP^3)$ is much subtler than in the surface case; see, for instance, the recent work of Jelisiejew, Ramkumar, and Sammartano \cite{JRS}.
Proposition~\ref{prop:planar-locus-degeneracy} identifies the planar locus with the reduction of the natural determinantal locus $D_3(\mathrm{ev})$, which for $n\geq 4$ is locally defined by the $4\times4$ minors of a matrix representing the evaluation map \eqref{eq:eval-map}.
\end{remark}

\section{Divisor-supported sheaves on a Fano threefold}\label{sec:divisors}

We now formulate the analogue of the planar story for sheaves supported on a smooth divisor.
Throughout this section $X$ is a smooth projective Fano threefold, $\iota\colon D\hookrightarrow X$ is a smooth divisor, and
\[
N:=N_{D/X}\cong \OO_D(D)
\]
is the normal bundle.
Fix a line bundle $L\in \Pic(D)$.
For $Z\in D^{[n]}$, set
\[
E_Z:=\I_Z\otimes L,
\qquad
F_Z:=\iota_*E_Z.
\]
The fixed-determinant moduli space of rank-one torsion-free sheaves $E$ satisfying $E^{\dual\dual}\cong L$ and $\operatorname{length}(L/E)=n$ is $D^{[n]}$.
The ambient threefold deformation theory is controlled by the following analogue of Proposition~\ref{prop:tangent-obstruction-fiber}.

\begin{proposition}\label{prop:general-ext}
For every $Z\in D^{[n]}$ there are exact sequences of trace-free Ext groups
\begin{equation}\label{eq:general-tangent}
0\longrightarrow \Ext^1_D(E_Z,E_Z)_0
\longrightarrow \Ext^1_X(F_Z,F_Z)_0
\longrightarrow H^0(D,N)
\longrightarrow 0
\end{equation}
and
\begin{equation}\label{eq:general-obstruction}
\Ext^2_X(F_Z,F_Z)_0\cong \Ext^1_D(E_Z,E_Z\otimes N).
\end{equation}
Moreover,
\[
\Ext^2_D(E_Z,E_Z\otimes N)=0,
\qquad
\Hom_D(E_Z,E_Z\otimes N)\cong H^0(D,N).
\]
\end{proposition}

\begin{proof}
Tensoring by $L$ does not affect internal Homs between rank-one torsion-free sheaves. Hence
\[
\Hom_D(E_Z,E_Z\otimes N)\cong \Hom_D(\I_Z,\I_Z\otimes N).
\]
As in the proof of Lemma~\ref{lem:surface-ext-vanishing}, $\mathcal Hom_D(\I_Z,\I_Z)\cong \OO_D$, hence
\[
\mathcal Hom_D(E_Z,E_Z\otimes N)\cong N
\]
and therefore
\[
\Hom_D(E_Z,E_Z\otimes N)\cong H^0(D,N).
\]
Next,
\[
\Ext^2_D(E_Z,E_Z\otimes N)
\cong
\Hom_D(E_Z,E_Z\otimes K_D\otimes N^{-1})^\dual
\cong
\Hom_D(E_Z,E_Z\otimes K_X|_D)^\dual
\]
by Serre duality and adjunction.
Since $-K_X|_D$ is ample and $E_Z$ is Gieseker stable on $D$, the target has strictly smaller reduced Hilbert polynomial than $E_Z$, so this Hom group vanishes.
Thus $\Ext^2_D(E_Z,E_Z\otimes N)=0$.

Apply Lemma~\ref{lem:ext-sequence} with $Y=X$, $S=D$, and $E=G=E_Z$.
The term $\Ext^2_D(E_Z,E_Z)_0$ vanishes because the moduli of rank-one torsion-free sheaves on $D$ with fixed determinant $L$ is $D^{[n]}$, which is smooth.
Taking trace-free parts therefore yields \eqref{eq:general-tangent} and \eqref{eq:general-obstruction}.
\end{proof}

\begin{lemma}\label{lem:kodaira-N}
For the normal bundle $N=\OO_D(D)$ one has
\[
H^i(D,N)=0\qquad (i>0).
\]
Consequently the class
\[
\Ttw_D^{[n]}(N):= \chi(D, N)[\mathcal{O}_{ D^{[n]} } ] - Rp_*R\mathcal Hom(\mathcal J,\mathcal J\otimes N)
\in K^0(D^{[n]})
\]
is represented by an honest rank-$2n$ vector bundle, where $\mathcal J$ is the universal ideal sheaf on $D\times D^{[n]}$.
\end{lemma}

\begin{proof}
Adjunction gives
\[
N-K_D=-K_X|_D,
\]
which is ample because $X$ is Fano.
By Kodaira vanishing, $H^i(D,N)=0$ for $i>0$.
The vanishing of $\Ext^2_D(E_Z,E_Z\otimes N)$ from Proposition~\ref{prop:general-ext} then shows that the $K$-class $\Ttw_D^{[n]}(N)$ is represented by the vector bundle with fiber $\Ext^1_D(E_Z,E_Z\otimes N)$.
It remains to compute the rank.
By Riemann--Roch,
\[
\chi_D(E_Z,E_Z\otimes N)=\chi(N)-2n.
\]
Since $\Hom_D(E_Z,E_Z\otimes N)=H^0(D,N)=\chi(N)$ and $\Ext^2_D(E_Z,E_Z\otimes N)=0$, we obtain
\[
\dim \Ext^1_D(E_Z,E_Z\otimes N)=2n.
\]
Thus the bundle has rank $2n$.
\end{proof}

\begin{definition}\label{def:rigid}
We call the divisor $D$ \emph{rigid} if $H^0(D,N)=0$.
Equivalently, the linear system $|D|$ is zero-dimensional.
\end{definition}

The numerical equivalence in Definition~\ref{def:rigid} follows from the exact sequence
\[
0\to \OO_X\to \OO_X(D)\to N\to 0
\]
and the vanishings $H^1(X,\OO_X)=0$ and $H^0(X,\OO_X)=\C$ on a Fano threefold.

\begin{theorem}\label{thm:rigid-component}
Assume that $D$ is rigid and that $H^1(D,\OO_D)=0$.
Then the divisor-supported stable-sheaf component
\[
\M_{D,L,n}(X):=\{\iota_*(\I_Z\otimes L): Z\in D^{[n]}\}
\]
is isomorphic to $D^{[n]}$.
Its standard truncated perfect obstruction theory for rank-zero stable sheaves has obstruction bundle $\Ttw_D^{[n]}(N)$ and virtual class
\begin{equation}\label{eq:rigid-vfc}
\vfc{\M_{D,L,n}(X)}=c_{2n}\Bigl(\Ttw_D^{[n]}(N)\Bigr)\cap [D^{[n]}].
\end{equation}
In particular, the virtual dimension is $0$.
\end{theorem}

\begin{proof}
Because $D$ is rigid, there are no first-order deformations of the support divisor inside $X$.
The vanishing $H^1(D,\OO_D)=0$ shows that $L$ has no infinitesimal deformations in $\Pic(D)$, so there are no additional Picard directions.
Thus Proposition~\ref{prop:general-ext} shows that the tangent space to the ambient stable-sheaf moduli along the divisor-supported locus is exactly the tangent space to $D^{[n]}$.
Hence the divisor-supported component is $D^{[n]}$.
The obstruction space at a point is identified with $\Ext^1_D(E_Z,E_Z\otimes N)$ by \eqref{eq:general-obstruction}, so Lemma~\ref{lem:kodaira-N} identifies the obstruction bundle with $\Ttw_D^{[n]}(N)$.
Since $D^{[n]}$ is smooth, the virtual class is the top Chern class of the obstruction bundle.
The rank computation in Lemma~\ref{lem:kodaira-N} shows that the virtual dimension is zero.
\end{proof}

\begin{theorem}\label{thm:general-rigid-series}
For a rigid divisor $D\subset X$ satisfying $H^1(D,\OO_D)=0$, one has
\begin{equation}\label{eq:general-rigid-series}
\sum_{n\ge 0}q^n\int_{\vfc{\M_{D,L,n}(X)}}1
=
\sum_{n\ge 0}q^n\int_{D^{[n]}}e\Bigl(\Ttw_D^{[n]}(N)\Bigr)
=
\prod_{m\ge 1}(1-q^m)^{-\delta(D,X)},
\end{equation}
where
\begin{equation}\label{eq:delta-def}
\delta(D,X):=c_2(X)\cdot D+D^3.
\end{equation}
\end{theorem}

\begin{proof}
By Theorem~\ref{thm:rigid-component}, the virtual numbers are the Euler classes of the twisted tangent bundles.
Applying Carlsson--Okounkov with $S=D$ and $L=N$ gives
\[
\sum_{n\ge 0}q^n\int_{D^{[n]}}e\Bigl(\Ttw_D^{[n]}(N)\Bigr)
=
\prod_{m\ge 1}(1-q^m)^{(N,K_D-N)-e(D)}.
\]
It remains to identify the exponent.
For a rank-two bundle $E$ and a line bundle $M$ on a surface,
\[
c_2(E\otimes M)=c_2(E)+c_1(E)c_1(M)+c_1(M)^2.
\]
Thus
\[
-(N,K_D-N)+e(D)=\int_D c_2\bigl(T_D\otimes N\bigr).
\]
Using the exact sequence
\[
0\to T_D\to T_X|_D\to N\to 0,
\]
one computes
\[
\int_D c_2\bigl(T_D\otimes N\bigr)=c_2(X)\cdot D+D^3.
\]
This is exactly \eqref{eq:delta-def}.
\end{proof}

\begin{remark}\label{rem:moving-vs-rigid}
The same fiber integral appears for moving divisors.
If a smooth family of divisors with normal bundle $N$ is parametrized by a smooth base of dimension $h^0(D,N)$, then Proposition~\ref{prop:general-ext} shows that the ambient virtual dimension is precisely the dimension of the support-divisor family.
The rigid case is characterized by the vanishing of these support directions.
The planar theory of \S\ref{sec:series} is the case $D=H\subset \PP^3$, where $h^0(D,N)=3$.
\end{remark}

\section{Examples of rigid divisors}\label{sec:examples}

\subsection{The exceptional divisor in $\operatorname{Bl}_p\PP^3$}
Let $\widetilde X=\operatorname{Bl}_p\PP^3$, and let $E\subset \widetilde X$ be the exceptional divisor.
Then $E\cong \PP^2$ and
\[
N_{E/\widetilde X}\cong \OO_{\PP^2}(-1).
\]
Hence $E$ is rigid.
Taking $L=\OO_E$, Theorem~\ref{thm:general-rigid-series} gives a divisor-supported virtual theory on $E^{[n]}$ with generating series
\[
\prod_{m\ge 1}(1-q^m)^{-\delta(E,\widetilde X)}.
\]
Now
\[
\delta(E,\widetilde X)=\int_E c_2\Bigl(T_{\PP^2}\otimes \OO_{\PP^2}(-1)\Bigr)=1,
\]
so
\begin{equation}\label{eq:blowup-series}
\sum_{n\ge 0}q^n\int_{\vfc{\M_{E,\OO_E,n}(\widetilde X)}}1
=
\prod_{m\ge 1}(1-q^m)^{-1}.
\end{equation}
Thus the rigid divisor series is, up to the usual prefactor, $\eta(\tau)^{-1}$.

\subsection{A rigid section in a Fano $\PP^1$-bundle}
Let
\[
S=\PP^1\times \PP^1,
\qquad
\mathcal E=\OO_S\oplus \OO_S(-1,-1),
\qquad
Y:=\PP_S(\mathcal E),
\]
where $\PP_S(\mathcal E)$ parametrizes one-dimensional quotients of $\mathcal E$.
Let $\pi\colon Y\to S$ be the projection and let $\xi=c_1(\OO_Y(1))$.
The canonical bundle formula for a projective bundle gives
\[
-K_Y=2\xi+\pi^*\OO_S(3,3),
\]
so $Y$ is Fano.
Let $D\subset Y$ be the section corresponding to the quotient
\[
\mathcal E\twoheadrightarrow \OO_S(-1,-1).
\]
Then $D\cong S$ and its normal bundle is
\[
N_{D/Y}\cong \OO_S(-1,-1).
\]
Hence $D$ is rigid.
For $L=\OO_D$, Theorem~\ref{thm:general-rigid-series} gives
\[
\sum_{n\ge 0}q^n\int_{\vfc{\M_{D,\OO_D,n}(Y)}}1
=
\prod_{m\ge 1}(1-q^m)^{-\delta(D,Y)}.
\]
A direct computation on $S=\PP^1\times \PP^1$ yields
\[
\delta(D,Y)=\int_S c_2\Bigl(T_S\otimes \OO_S(-1,-1)\Bigr)=2,
\]
so
\begin{equation}\label{eq:p1p1-series}
\sum_{n\ge 0}q^n\int_{\vfc{\M_{D,\OO_D,n}(Y)}}1
=
\prod_{m\ge 1}(1-q^m)^{-2}.
\end{equation}
Thus the rigid divisor series is, again up to the usual prefactor, $\eta(\tau)^{-2}$.

\begin{remark}\label{rem:gs}
Theorem~\ref{thm:general-rigid-series} and the two examples above are closely parallel to the surface-supported theories studied in \cite{GSmod,GSlinear}.
What is special here is that the moving and rigid cases are governed by one and the same bundle, namely the twisted tangent bundle associated to the normal bundle of the support divisor.
\end{remark}

\section*{Acknowledgements}
The author thanks Dominic Joyce for inspiring this project. The author also thanks Artan Sheshmani for many helpful conversations on Donaldson--Thomas theory. 

\section*{AI acknowledgements}
The author is greatly indebted to many conversations with ChatGPT 5.6 Pro which made this paper possible.

\end{document}